\documentclass[11pt]{amsart}
 \usepackage{amssymb,amsmath,amsfonts,epsfig,latexsym,xypic, hyperref}

 \newtheorem{theorem}{Theorem}[section]

\newtheorem{proposition}[theorem]{Proposition}

\newtheorem*{WET}{W\l odarczyk's Embedding Theorem}

\theoremstyle{definition}
\newtheorem{remark}[theorem]{Remark}

\def\M{\ensuremath{\mathbb{M}}}

\def\Z{\ensuremath{\mathbb{Z}}}
\def\Q{\ensuremath{\mathbb{Q}}}

\def\RR{\ensuremath{\mathbb{R}}}
\def\R{\ensuremath{\mathbf{R}}}

\def\cU{\ensuremath{\mathcal{U}}}

\def\an{\mathrm{an}}

\def\<{\ensuremath{\langle}}
\def\>{\ensuremath{\rangle}}

\def\bx{\mathbf{x}}

\def\cS{\mathcal{S}}

\DeclareMathOperator{\Div}{Div}

\DeclareMathOperator{\Hom}{Hom}

\DeclareMathOperator{\Spec}{Spec}

\DeclareMathOperator{\trop}{Trop}
\DeclareMathOperator{\Trop}{\bf Trop}

\DeclareMathOperator{\val}{val}

\begin{document}

\title{Limits of tropicalizations}

\author{Tyler Foster}
\author{Philipp Gross}
\author{Sam Payne}

\address{Yale University
Mathematics Department\\
10 Hillhouse Ave\\
New Haven, CT 06511 \\ U.S.A.}
\email{tyler.foster@yale.edu}

\address{
Fakult\"at f\"ur Mathematik \\
Universit\"atsstr. 1 \\
40225 DŸsseldorf \\
Germany}

\email{gross@math.uni-duesseldorf.de}

\address{Yale University
Mathematics Department\\
10 Hillhouse Ave\\
New Haven, CT 06511 \\ U.S.A.}
\email{sam.payne@yale.edu}

\begin{abstract}
We give general criteria under which the limit of a system of tropicalizations of a scheme over a nonarchimedean field is homeomorphic to the analytification of the scheme.  As an application, we show that the analytification of an arbitrary closed subscheme of a toric variety is naturally homeomorphic to the limit of its tropicalizations, generalizing an earlier result of the third author for quasiprojective varieties.
\end{abstract}

\maketitle

\section{Introduction}

Let $K$ be a field that is complete with respect to a nonarchimedean valuation, and let $X$ be a scheme of finite type over $K$.  The analytification $X^\an$, in the sense of Berkovich \cite{Berkovich90}, is a locally compact, locally contractible Hausdorff space that admits a natural proper, continuous projection onto the tropicalization of the image of any closed embedding of $X$ into a toric variety.  These projections are compatible with maps between tropicalizations induced by toric morphisms.  Therefore, for any inverse system $\cS$ of toric embeddings of $X$ there is a natural map to the corresponding limit of tropicalizations
\[
\pi_\cS : X^\an \rightarrow \varprojlim_{\iota \in \cS} \trop(X, \iota).
\]
Our first main technical result is a sufficient condition on the system $\cS$ for $\pi_\cS$ to be a homeomorphism.  

\begin{theorem} \label{thm:sufficient}
Let $\cS$ be a system of toric embeddings that contains finite products.  Suppose there is an affine open cover $X = U_1 \cup \cdots \cup U_r$ with the following property, for $1 \leq i \leq r$:
\begin{itemize}
\item For any nonzero regular function $f \in K[U_i]$, there is an embedding $\iota \in \cS$ such that $U_i$ is the preimage of a torus invariant open set and $f$ is the pullback of a monomial.
\end{itemize}
Then $\pi_\cS$ is a homeomorphism.
\end{theorem}

\noindent The closed embeddings of a quasiprojective variety into torus invariant open subsets of projective space satisfy $(\star)$.  This is proved using properties of ample line bundles in \cite[Lemma~4.3]{analytification}.  The main result of that paper, which says that $\pi_\cS$ is a homeomorphism when $X$ is a quasiprojective variety and $\cS$ is the system of all of its closed embeddings into quasiprojective toric varieties, then follows easily.   Theorem~\ref{thm:sufficient} is significantly more general because it applies equally to schemes that are reducible, non reduced, and non quasiprojective, and because it allows some flexibility in the choice of the system $\cS$.   For instance, if $X$ is a quasiprojective variety, then one could replace the system of all closed embeddings into quasiprojective toric varieties with a smaller system, such as embeddings into invariant open subsets of products of projective spaces, or a larger system, such as all closed embeddings into toric varieties.

Theorem~\ref{thm:sufficient} makes precise the idea that the analytification of a scheme with enough toric embeddings should be homeomorphic to the limit of its tropicalizations.  Ample line bundles produce enough toric embeddings on quasiprojective varieties, but there are many toric varieties and subschemes of toric varieties that have no nontrivial line bundles at all, for which such methods do not apply.  See \cite[Example~3.2]{Eikelberg92}, \cite[pp. 25--26, 72]{Fulton93}, and the examples studied in \cite{branchedcovers}.  Nevertheless, in Section~\ref{sec:embeddings} we show that any scheme that admits a single closed embedding into a toric variety also admits enough embeddings to satisfy $(\star)$.  

\begin{theorem} \label{thm:subscheme}
Let $X$ be a closed subscheme of a toric variety over $K$.  Then the natural map from $X^\an$ to the limit of the inverse system of tropicalizations of all toric embeddings of $X$ is a homeomorphism.
\end{theorem}

In order to prove Theorem~\ref{thm:subscheme}, we must show that an arbitrary closed subscheme of a toric variety admits many closed embeddings into toric varieties.  The subject of closed embeddings in toric varieties has been studied intensively by several authors \cite{Wlodarczyk93, Cox95b, Hausen00, Hausen02, BerchtoldHausen02, HausenSchroer04}.  We recall the first fundamental result in this subject.

\begin{WET}
Let $X$ be a normal variety.  Then $X$ admits a closed embedding into a toric variety if and only if any two points of $X$ are contained in an affine open subvariety.
\end{WET}

\noindent The two point condition in W\l odarczyk's Embedding Theorem is necessary because a toric variety is covered by invariant affine opens and the union of any two invariant affine opens in a toric variety is quasiprojective.  Therefore, if $X$ is a closed subscheme of a toric variety then any two points of $X$ are contained in a quasiprojective open subset of $X$, and hence share an affine open neighborhood.  

W\l odarczyk's proof of this theorem gives a flexible algorithm to construct embeddings with desirable properties.  For instance, he shows that if $X$ is smooth then the embedding can be constructed so that the ambient toric variety is also smooth.  However, the theorem does not apply to non normal schemes, and the arguments in the proof do not generalize easily to this case; it remains an open problem to characterize the non normal schemes that admit closed embeddings into toric varieties, even if one allows embeddings into non normal toric varieties.  

Our proof of Theorem~\ref{thm:subscheme} does not apply the theory of toric embeddings to $X$ directly.  Instead, given a closed subscheme $X$ of a toric variety $Y$, we apply the algorithm from W{\l}odarczyk's proof of his embedding theorem to the ambient toric variety and show that $(\star)$ holds for the system of toric embeddings of $X$ that factor through closed embeddings of $Y$.  The same then holds for any larger system of embeddings that contains finite products, including the inverse system of all toric embeddings of $X$.

Alternative approaches to understanding the topology of analytifications of varieties using limits of polyhedral complexes involve skeletons of semistable or pluristable formal models \cite{Berkovich04, KontsevichSoibelman06}, or spaces of definable types \cite{HrushovskiLoeser10}.  Such methods produce inverse systems with desirable properties, e.g., in which every map of polyhedra is a strong deformation retraction, but require far more technical machinery.

\noindent \textbf{Acknowledgments.}  We thank Melody Chan and the referee for pointing out an error in an earlier version of this paper.  The work of SP is partially supported by NSF DMS-1068689 and NSF CAREER DMS-1149054.

\section{Preliminaries}

Throughout, we work over a field $K$ that is complete with respect to a nonarchimedean valuation, which may be trivial.  Let $N \cong \Z^n$ be a lattice, let $M = \Hom(N,\Z)$ be its dual lattice, and let $T = \Spec K[M]$ be the torus with character lattice $M$.

Affine toric varieties with dense torus $T$ correspond naturally and bijectively with rational polyhedral cones in the real vector space $N_\RR = N \otimes_\Z \RR$.  For such a rational polyhedral cone $\sigma$, we write $S_\sigma$ for the additive monoid of lattice points in $M$ that are nonnegative on $\sigma$.  Then the corresponding affine toric variety is
\[
U_\sigma = \Spec K[S_\sigma].
\]
We write $u$ for a lattice point in $M$ and $\bx^u$ for the corresponding regular function and refer to the scalar multiples $a\bx^u$, for $a \in K^*$, as monomials.

The vector space $N_\RR$ is the usual tropicalization of the dense torus $T$.  Tropicalizations of more general toric varieties were studied implicitly by Thuillier, as skeletons of analytifications  \cite{Thuillier07}, and explicitly by Kajiwara \cite{Kajiwara08}, before being applied to the development of explicit relations between tropical and nonarchimedean analytic geometry in \cite{analytification, BPR11, Rabinoff12}, to which we refer the reader for further details.  The tropicalization of the affine toric variety $U_\sigma$ is the space of monoid homomorphisms
\[
N_\R(\sigma) = \Hom(S_\sigma, \R),
\]
where $\R$ is the additive monoid $\RR \cup \{ + \infty \}$.  This space of monoid homomorphisms is a partial compactification of $N_\RR \cong \Hom(S_\sigma, \RR)$, just as $U_\sigma$ is a partial compactification of $T$.

Recall that the analytification $U^\an$ of an affine scheme over $K$ is the set of ring valuations
\[
\eta: K[U] \rightarrow \R,
\]
that extend the given valuation on $K$, equipped with the subspace topology for the natural inclusion of $U^\an$ in $\R^{K[U]}$.  This is the coarsest topology such that the map to $\R$ induced by each function $f$ in $K[U]$ is continuous.  The tropicalization map for the affine variety $U_\sigma$ is the natural projection
\[
\trop: U_\sigma^\an \rightarrow N_\R(\sigma)
\]
taking a valuation $\eta$ to the monoid homomorphism $[u \mapsto \eta(\bx^u)]$.  This construction generalizes to arbitrary toric varieties and their subvarieties as follows.

Let $\Delta$ be a fan in $N_\RR$, and let $Y_\Delta$ be the corresponding toric variety with dense torus $T$, so $Y_\Delta$ is the union of the affine toric varieties $U_\sigma$ corresponding to cones $\sigma$ in $\Delta$.  Then $Y_\Delta^\an$ is the union of the corresponding subsets $U_\sigma^\an$, which are again open \cite[Proposition~3.4.6]{Berkovich90}, and the spaces $N_\R(\sigma)$ glue together to give $N_\R(\Delta)$ with a natural map
\[
\trop: Y_\Delta^\an \rightarrow N_\R(\Delta)
\]
whose restriction to each $U_\sigma^\an$ is the map $\trop$ described above.  This global tropicalization map is continuous, surjective, and proper in the sense that the preimage of a compact set is compact \cite[Lemma~2.1]{analytification}.  In particular, it is a closed map.  Now, if $X$ is a scheme of finite type over $K$ and $\iota: X \hookrightarrow Y_\Delta$ is a closed embedding, then $\iota^\an: X^\an \hookrightarrow Y_\Delta^\an$ is also a closed embedding \cite[Proposition~3.4.7]{Berkovich90}.  Therefore, the image $\trop(X, \iota)$ of $X^\an$ is closed in $N_\R(\Delta)$.

Tropicalization maps are covariantly functorial; an equivariant morphism of toric varieties $\varphi: Y_\Delta \rightarrow Y_{\Delta'}$ induces a continuous map $\trop(\varphi): N_\R(\Delta) \rightarrow N_\R(\Delta')$ that forms a commutative diagram with the tropicalization maps and the analytification of $\varphi$.  Now, suppose
\[
\iota: X \hookrightarrow Y_\Delta \mbox{ and } \jmath: X \hookrightarrow Y_{\Delta'}
\]
are closed embeddings.  We say that $\varphi$ is a morphism of embeddings if $\varphi \circ \iota = \jmath$.  In this case, $\trop(\varphi)$ maps $\trop(X,\iota)$ onto $\trop(X, \jmath)$.

A system $\cS$ of toric embeddings of a scheme $X$ of finite type over $K$ is a diagram of closed embeddings of $X$ in toric varieties, with arrows given by morphisms of toric embeddings.  By the functoriality properties of tropicalization discussed above, there is a corresponding diagram of tropicalizations, whose objects are the spaces $\trop(X, \iota)$ for $\iota \in \cS$, and whose arrows are the tropicalizations of the arrows in $\cS$.  We say that $\cS$ contains finite products if, for any embeddings $\iota$ and $\iota'$ in $\cS$, the product $\iota \times \iota'$ is in $\cS$, along with the natural projections from $\iota \times \iota'$ to $\iota$ and $\iota'$.

\section{Limits of systems of tropicalizations}

We now prove Theorem~\ref{thm:sufficient}, which says that a system $\cS$ of toric embeddings of $X$ that contains finite products and satisfies $(\star)$ induces a homeomorphism $\pi_\cS$ from $X^\an$ to the limit of the corresponding system of tropicalizations.

\begin{proposition} \label{prop:surj}
Let $\cS$ be a system of toric embeddings of $X$ that contains finite products.  Then $\pi_\cS$ is surjective.
\end{proposition}

\begin{proof}
Let $x = (x_\iota)_{\iota \in \cS}$ be a point in the inverse limit over $\iota \in \cS$ of $\trop(X, \iota)$.  We must show that the fiber
\[
\pi_\cS^{-1}(x) = \bigcap_{\iota \in \cS} \trop^{-1}(x_\iota)
\]
is nonempty.

Fix some $\jmath$ in $\cS$.  Then $\trop^{-1}(x_\jmath)$ is a nonempty compact subset of $X^\an$ that contains $\pi_\cS^{-1}(x)$.  For any finite collection of embeddings $\iota_1, \ldots \iota_s$ in $\cS$, the intersection
\[
\trop^{-1}(x_{\iota_1}) \cap \cdots \cap \trop^{-1}(x_{\iota_s}) \cap \trop^{-1}(x_\jmath)
\]
is nonempty, since it contains $\trop^{-1}(x_{\iota_1 \times \cdots \times \iota_s \times \jmath})$.  Since each of these finite intersections inside the compact set $\trop^{-1}(x_\jmath)$ is nonempty, it follows that the full intersection $\bigcap_{\iota \in \cS} \trop^{-1}(x_\iota)$ is nonempty, as required.
\end{proof}

\begin{proposition}  \label{prop:inj}
Let $\cS$ be a system of toric embeddings of $X$ that satisfies $(\star)$.  Then $\pi_\cS$ is injective.
\end{proposition}

\begin{proof}
Suppose $\cS$ satisfies $(\star)$ and let $\eta$ and $\eta'$ be distinct points in $X^\an$.   We must show that there is an embedding $\iota$ in $\cS$ such that the images of $\eta$ and $\eta'$ are distinct in $\trop(X, \iota)$.  Fix an affine open cover $U_1, \ldots, U_r$ of $X$ such that, for $1 \leq i \leq r$ and $f \in K[U_i]$, there is a toric embedding of $X$ such that $U_i$ is the preimage of a torus invariant open set and $f$ is the pullback of a monomial.  

Choose $i$ such that $\eta$ is in $U_i^\an$, and suppose $\eta'$ is not in $U_i^\an$.  Let $\iota$ be a toric embedding such that $U_i$ is the preimage of a torus invariant open subset $U_{\sigma_1} \cup \cdots \cup U_{\sigma_s}$. Then $\pi_\iota(\eta)$ is in the open subset $N_\R(\sigma_1) \cup \cdots \cup N_\R(\sigma_s)$ of $N_\R(\Delta)$, but $\pi_\iota(\eta')$ is not.

Otherwise, $\eta$ and $\eta'$ are distinct points in $U_i^\an$, corresponding to distinct ring valuations on $K[U_i]$.  Let $f$ be a regular function in $K[U_i]$ such that $\eta(f)$ is not equal to $\eta'(f)$.  Choose a toric embedding $\jmath$ of $X$ such that $U_i$ is the preimage of an invariant open set and $f$ is the pullback of a monomial.  Say $U_\sigma$ is an invariant affine open whose analytification contains the image of $\eta$.  If $U_\sigma^\an$ does not contain the image of $\eta'$ then $\pi_\iota(\eta)$ is in $N_\R(\sigma)$ but $\pi_\iota(\eta')$ is not.

It remains to consider the case where $U_\sigma^\an$ contains both $\iota(\eta)$ and $\iota(\eta')$.  Then there is a nonzero scalar $a \in K^*$ such that $af$ is the pullback of a character $\bx^u$, and $\pi_\jmath(\eta)$ and $\pi_\jmath(\eta')$ are monoid homomorphisms from $S_\sigma$ to $\R$ that take $u$ to $\eta(f) + \val(a)$ and $\eta'(f) + \val(a)$, respectively.  In particular, $\pi_\jmath(\eta)$ and $\pi_\jmath(\eta')$ are distinct, as required.
\end{proof}

Finally, we show that existence of finite products and property $(\star)$ are together enough to guarantee that $\pi_\cS$ is a homeomorphism.

\begin{proof}[Proof of Theorem~\ref{thm:sufficient}]
Let $\cS$ be a system of toric embeddings of $X$ that contains finite products and satisfies $(\star)$.  By Propositions~\ref{prop:surj} and \ref{prop:inj} the induced map
\[
\pi_\cS: X^\an \rightarrow \varprojlim_{\iota \in \cS} \trop(X, \iota)
\]
is a continuous bijection.

Choose an open cover $U_1, \ldots, U_r$ verifying property $(\star)$, let $f$ be a regular function in $K[U_i]$, and let $\iota$ be a toric embedding such that $U_i$ is the preimage of an invariant open subset $U$ and $f$ is the pullback of a monomial.  Say $U$ is the union of invariant affine opens $U_{\sigma_1} \cup \cdots \cup U_{\sigma_s}$, and let $\cU$ be the preimage of $N_\R(\sigma_1) \cup \cdots \cup N_\R(\sigma_s)$ in $\varprojlim \trop(X, \iota)$.  So $\cU$ is exactly the image of $U_i^\an$ under $\pi_\cS$.  Now, suppose $f$ is the pullback of the monomial $a\bx^u$.  The topology on $N_\R(\sigma_1) \cup \cdots \cup N_\R(\sigma_s)$ is the coarsest such that the map
\[
\pi_f : N_\R(\sigma_1) \cup \cdots \cup N_\R(\sigma_s) \rightarrow \R
\]
taking a monoid homomorphism $\phi: S_{\sigma_i} \rightarrow \R$ to $\phi(u) + \val(a)$ is continuous, for all such $f$.  Therefore, the composite map $\cU \rightarrow \R$ is continuous  as well.  The map $U_i^\an \rightarrow \R$ taking a valuation $\eta$ to $\eta(f)$ is also continuous, and factors through $\pi_f$.  Moreover, the topology on $U_i^\an$ is the coarsest such that all such maps to $\R$ are continuous.  It follows that $\pi_\cS$ restricts to a homeomorphism from $U_i^\an$ to $\cU$.  Since $\pi_\cS$ is bijective and the $U_i^\an$ form an open cover of $X^\an$, it follows that $\pi_\cS$ is a homeomorphism, as required.
\end{proof}

\section[Proof of Theorem~1.2]{Proof of Theorem~\ref{thm:subscheme}} \label{sec:embeddings}

For reducible and non reduced schemes, it is natural to look for a condition such as $(\star)$ phrased in terms of affine covers and regular functions.  However, for irreducible varieties it is often easier and more convenient to avoid dealing with open covers and domains of regularity by working directly with rational functions.  We will prove Theorem~\ref{thm:subscheme} using the following observation: when a scheme $X$ is embedded in an irreducible variety $Y$, one can use rational functions on $Y$ to produce embeddings of $X$ that factor through embeddings of $Y$.

\begin{proposition} \label{prop:rationalFunctions}
Let $Y$ be a variety over $K$ and let $\cS$ be a system of toric embeddings of $Y$ that contains finite products.  Suppose that for each rational function $f$ in $K(Y)^*$ there is an embedding $\iota$ in $\cS$ such that $f$ is the pullback of a monomial.  Then the restriction of $\cS$ to $X$ satisfies $(\star)$.  \end{proposition} 

\noindent In particular, when the hypotheses of the proposition are satisfied then the restriction of $\pi_\cS$ to $X^\an$ is a homeomorphism onto $\displaystyle{\varprojlim_{\iota \in \cS}} \Trop(X, \iota)$.

\begin{proof}
We show that the restriction of $\cS$ to $X$ satisfies $(\star)$.  Let $V_1, \ldots, V_r$ be an affine open cover of $Y$, and let
\[
U_i = X \cap V_i,
\]
so $U_1, \ldots, U_r$ is an affine open cover of $X$.  Since $X$ is closed in $Y$, the restriction map on coordinate rings $K[V_i] \rightarrow K[U_i]$ is surjective.  Let $f \in K[U_i]$ be a regular function, and choose $g \in K[V_i]$ that restricts to $f$.  We must show that there is an embedding $\iota \in \cS$ such that $U_i$ is the preimage in $X$ of an invariant open subset $U$ and $f$ is the pullback of a monomial.

For $j \neq i$, choose functions
\[
g_{j1}, \ldots, g_{js}
\]
in $K[V_j]$ that generate the ideal of the closed subvariety $Y \smallsetminus V_i$ on $V_j$.  By hypothesis, there exists an embedding $\iota$ in $\cS$ such that $g$ is the pullback of a monomial, and embeddings $\iota_{jk}$ such that $g_{jk}$ is the pullback of a monomial.  We consider the product embedding
\[
\jmath = \iota \times \Big( \prod_{j,k} \iota_{jk} \Big),
\]
which is also in $\cS$.

By construction, the image under $\jmath$ of the closed subvariety $Y \smallsetminus V_i$ is cut out by vanishing loci of monomials, and hence is the preimage of a torus invariant closed set.  Therefore $V_i$ is the preimage in $Y$ of an invariant open subset $U$, and $g$ is the pullback of a monomial that is regular on $U$. Restricting $\jmath$ to $X$, we see that $U_i$ is the preimage of $U$ and $f$ is the pullback of a monomial.  So, the restriction of $\cS$ to $X$ satisfies $(\star)$ and the theorem follows.
\end{proof}

To prove Theorem~\ref{thm:subscheme}, we will apply Proposition~\ref{prop:rationalFunctions} to a system of embeddings of $Y$ in toric varieties that is generated using W\l odarczyk's algorithm.  We begin by briefly recalling the outline of this procedure, from \cite{Wlodarczyk93}.

\bigskip

Let $Y$ be a normal variety over $K$, and suppose $\iota: Y \hookrightarrow Y_\Delta$ is a toric embedding whose image meets the dense torus $T \subset Y_\Delta$.  Then the pullbacks of the characters on the dense torus form a finitely generated group $\M$ of rational functions on $Y$, and the set of fibers over the generic points of the $T$-invariant divisors of $Y_\Delta$ is a finite collection $\Div(\M)$ of Weil divisors in $Y$ with no common components.  The pair $(\M, \Div(\M))$ satsifies four axioms that encode familiar properties satisfied by the monomial functions and intersections with coordinate hyperplanes in any nondegenerate embedding of $Y$ into projective space \cite[\S 3.1]{Wlodarczyk93}.  The first axiom says that the divisor of any function in $\M$ is a $\Z$-linear combination of divisors in $\Div(\M)$.  The second and third axioms are a separation and a saturation axiom, respectively.  The final axiom says that, for any point $p \in Y$, the subsemigroup $S_p$ of rational functions in $\M$ that are regular at $p$ is finitely generated, that the open set $U_p$ where all functions in $S_p$ are regular is affine, and that $S_p$ generates the coordinate ring $K[U_p]$.

A pair $(\M, \Div(\M))$ satisfying these four axioms is called an \emph{embedding system of functions}.  The axioms ensure that the dual cones $\sigma_p \subset \Hom(\M,\mathbb{R})$, consisting of linear functions that are nonnegative on the semigroups $S_p$ for $p \in Y$, form a fan $\Delta$, and that $\M$ is the pullback of the group of characters under a toric embedding $\iota: Y \rightarrow Y_\Delta$ whose image meets the dense torus.

The heart of W{\l}odarczyk's proof is the construction of an embedding system of functions $(\M,\Div(\M))$, which is accomplished through a ten-step procedure.  Some steps of this procedure are algorithmic, but others involve choices.  In particular, in the first step, one chooses a finite, Zariski open cover $\mathcal{U}$ of $Y$ such that any two points of $Y$ lie in some $U$ in $\mathcal{U}$, and in the second step one chooses a finite generating set for each of the coordinate rings $K[U]$. The flexibility coming from these choices yields the following embedding theorem.  This statement is stronger than what is needed to prove Theorem~\ref{thm:subscheme}, but will be useful for intended future applications.

\begin{theorem}  \label{thm:strongEmbedding}
Let $Y$ be a normal variety in which any two points are contained in an affine open subvariety.  Let $U_1, \ldots, U_r$ be affine open subvarieties of $Y$, and let $R_j$ be a finite collection of regular functions on $U_j$, for $1 \leq j \leq r$.  Then there exists a closed embedding $\iota: Y \hookrightarrow Y_{\Delta}$ such that $U_j$ is the preimage of an invariant affine open subset $U_{\sigma_j}$ and each element of $R_j$ is the pullback of a character that is regular on $U_{\sigma_j}$, for $1 \leq j \leq r$.
\end{theorem}

\begin{proof}
We apply the ten-step procedure to our variety $Y$. Using W{\l}odarczyk's notation, our initial indexing set $A$ will be the singleton $A=\{a\}$, with $Y_a = Y$.  In Step 1 of his algorithm, include $U_1, \ldots, U_r$ among the finite set of affine opens $\{U_{ab}\}_{b \in B}$ such that every pair of points in $Y$ lies in some $U_{ab}$.  In Step 2, when choosing a finite set of generators $\{f_{abc}\}_{c \in C}$ for each coordinate ring $K[U_{ab}]$, make sure that the generating set for $K[U_j]$ includes $R_j$, for each $1 \leq j \leq r$.  Then upon completing the remaining Steps 3-10 of W{\l}odarczyk's algorithm, one obtains an embedding system of functions $(\M,\Div(\M))$ with the property that each $f$ in $R_j$ appears in $\M$, for each $1 \leq j \leq r$, and each of the affine open subvarieties $U_1, \ldots, U_r$ of $Y$ appears as one of the affine opens $U_{\mathfrak{p}}\subseteq Y$ constructed from $(\M,\Div(\M))$. The embedding $\iota: Y \hookrightarrow Y_{\Delta}$ built from $(\M,\Div(\M))$ is then a closed toric embedding that verifies the theorem.
\end{proof}

We now apply Proposition~\ref{prop:rationalFunctions} to prove Theorem~\ref{thm:subscheme}.

\begin{proof}[Proof of Theorem~\ref{thm:subscheme}]
Let $X$ be a closed subscheme of a toric variety $Y$. To prove Theorem~\ref{thm:subscheme}, it is enough to show that the system of embeddings of $X$ that factor through embeddings of $Y$ satisfies $(\star)$.  By Proposition~\ref{prop:rationalFunctions}, it is enough to show that for each rational function $f$ in $K(Y)^*$ there is a toric embedding of $Y$ such that $f$ is the pullback of a monomial.  Now, let $U$ be an affine open subset of $Y$ on which $f$ is regular.  By Theorem~\ref{thm:strongEmbedding} there is a toric embedding such that $U$ is the preimage of a torus invariant affine open subset  and $f$ is the pullback of a monomial, as required.
\end{proof}

\begin{remark}
The proof of Theorem~\ref{thm:subscheme} can be adapted to give stronger information in various contexts.  For instance, the targets of the toric embeddings of $Y$ produced by W\l odarczyk's construction are smooth or $\Q$-factorial if and only if $Y$ is so.  It follows that, if $X$ is a scheme that admits a closed embedding into a single toric variety that is smooth or $\Q$-factorial, then its analytification is naturally homeomorphic to the limit of the tropicalizations of all of its embeddings into toric varieties with the same property.  Also, from W{\l}odarczyk's Embedding Theorem, we deduce that if $X$ is a normal variety over an algebraically closed field in which any two points are contained in an affine open subvariety, then the analytification of $X$ is naturally homeomorphic to the limit of the tropicalizations of its toric embeddings.
\end{remark}

\bibliography{math}
\bibliographystyle{amsalpha}

\end{document}